\documentclass{article}

\usepackage{amssymb,amsthm}
\usepackage{amscd}
\usepackage{xcolor}
\usepackage[all]{xy}
\input xy
\xyoption{all}

\newtheorem*{theorem*}{Theorem}
\theoremstyle{plain}

\newtheorem{proposition}{Proposition}

\title{The approximation property and Lipschitz mappings on Banach spaces}
\author{Pilar Rueda and Enrique A. S\'anchez-P\'erez}
\date{}

\begin{document}
\maketitle

\begin{abstract}
We present an overview to the approximation property, paying especial attention to the recent results relating  the approximation property to ideals of linear operators and Lipschitz ideals.
 We complete the paper with some new results on  approximation of Lipschitz mappings and their relation to linear operator ideals.
\end{abstract}

\section{The approximation property}

The approximation property is a fundamental property in Functional Analysis. Although the widest context where it has been studied are locally convex spaces, it has fruitfully succeeded  on Banach spaces theory. A locally convex space (l.c.s.) $E$ has the {\it approximation property} (AP for short) if  the identity on $E$ can be approximated uniformly on precompact subsets of $E$ by finite rank operators. Denoting $Id_E$ the identity of $E$, $E'$ the topological dual of $E$, $E'\otimes$ the tensor product of $E'$ and $E$ and ${\mathcal L}(E)$ the space of all continuous operators on $E$, the AP can be written as $Id_E\in \overline {E'\otimes E}^{\tau_c}$ in ${\mathcal L}(E)$, where $\tau_c$ is the topology of uniform converge on precompact subsets of $E$.
 In \cite[p. 108]{Sc}
 one can read:
\begin{quote}
But it is not known whether in every l.c.s. $E$ the identity map  can be approximated, uniformly on every precompact set in $E$, by continuous linear maps of finite rank. The question whether or not this is true constitutes the {\it approximation problem.}
\end{quote}
Around 1970 every known Banach  space (even every known l.c.s.) had the  AP. Gr\"othendieck  proved in 1955 that if every Banach space has the AP then every l.c.s. has the  AP. This result made of special interest the  study of the AP for Banach spaces. At that time it was known that every Banach space with a basis had the AP. So, the problem of the basis posed by Banach in 1932 hit the  scene: {\it  Does every separable Banach space have a Schauder basis?}. Karlin proved in 1948 that there exist separable Banach spaces without an unconditional basis, e.g. $C([0,1])$. However, a complete answer to the basis problem was unknown. So, the relation between both problems, the approximation problem and the problem of the basis, was very strong. To make things more interesting, Gr\"othendieck  also proved that if every  Banach space had the  AP then every separable Banach space had a basis. These results spread the interest of the AP on  Banach spaces, that became the natural setting where the theory was mainly developed.  Of course all this was before 1973, when Enflo  solved the Approximation Problem by giving a celebrated example of a separable Banach space without a basis and without the AP.
After Enflo's example, many new ones appeared in the literature. For instance, the interested reader can find details of the following ones  in \cite{Ca}:
 \begin{itemize}
\item Figiel (unpublished) and Davie (1973/75) found an example of a subspace  $S$ of $\ell_p$, $2<p$, without the  AP.
\item Szankowski (1978) of a subspace $S$ of $\ell_p$, $1\leq p<2$, without the (compact) AP.
\item Johnson (1979/80) gave a  Banach space whose subspaces have all the  AP but some of them have no basis.
\item Szankowski (1981) proved that $\mathcal L(\ell_2;\ell_2)$ does not have the AP.
\item Pisier (1983/85) found an infinite dimensional  Banach space such that $X\hat \otimes_\pi X=X\hat\otimes_\varepsilon X$ algebraically  and topologically. Such $X$ does not have the  AP.
\end{itemize}

Although the definition of the AP is estated just for the identity been approximated, it is  well-known that not only the identity can be approximated. If   $E$ is a l.c.s. with topological dual  $E'$ then, the following are equivalent:
 \begin{enumerate}
 \item $Id_E\in \overline {E'\otimes E}^{\tau_c}$ in ${\mathcal L}(E)$ (i.e. $E$ has the  AP)
 \item ${\mathcal L}(E)=\overline {E'\otimes E}^{\tau_c}$.
 \item  For all l.c.s. $F$, $\mathcal L(E;F)=\overline {E'\otimes F}^{\tau_c}$.
\item For all l.c.s. $F$, $\mathcal L(F;E)=\overline {F'\otimes E}^{\tau_c}$.
 \end{enumerate}

Also compact operators can be approximated uniformly, not only on compact sets, but also on bounded sets. This was proved by Grothendieck  in 1955.
Let $E$ be a  Banach space with strong dual $E'$ and let $\tau_b$ be the topology of uniform convergence on bounded sets of $E$.
 \begin{itemize}
 \item $E$ has the AP $\iff$ for every Banach space $F$, $\mathcal K(F;E)=\overline {F'\otimes E}^{\tau_b}$ in $\mathcal L(F;E)$.
 \item $E'$ has the  AP $\iff$ for every Banach space $F$, $\mathcal K(E;F)=\overline {E'\otimes F}^{\tau_b}$ in $\mathcal L(E;F).$
 \end{itemize}
where $\mathcal K(E;F)$ are all compact operators from  $E$ into $F$.

However, it is still an open problem to know that a Banach space has the AP whenever
  $\mathcal K(E;E)=\overline {E'\otimes E}^{\tau_b}$ in $\mathcal L(E;E)$.


\section{The approximation property and ideals of operators}
In the last decade linear operator ideals have hit the study of the approximation property. The main idea is to consider variants of the approximation property when we consider not only bounded or compact operators, but any linear operator belonging to some given operator ideal. In these variants a suitable operator topology related to the ideal is required. In \cite{BeBo}, it is considered an approximation property  where every linear operator $T \in \mathcal L(E,E)$ can be approximated uniformly on compact subsets of $E$ by  operators in $\mathcal I(E,E)$. In \cite{c,DOPS,oja12,r} the operators from the ideal are approximated by finite rank operators. Besides, replacing compact sets  by another class of sets with some kind of compactness related to the ideal $\mathcal I$ has also been considered. The new class of sets is formed by   $\mathcal I$-compact sets. This notion was introduced by Carl and Stephani  in \cite{CaSt} and the related approximation property has been studied in \cite{DOPS, LaTu,LaTu3}. All these ideas have been unified in \cite{BeBo1} where  the concept of ideal topology  has been introduced. Given ${\mathcal I,J}$ two ideals of linear operators,  the unifying approximation property given in \cite{BeBo1} reads as follows: a Banach space $E$ is said to have the $({\mathcal I,J} , \tau)$-approximation property  if $E$-valued operators belonging
to the operator ideal ${\mathcal I}$ can be approximated, with respect to the ideal topology $\tau$, by operators
belonging to the operator ideal ${\mathcal J}$.

Let us recall some concepts related to the tandem approximation property/operator ideals.
Let $\mathcal I$ be an operator ideal, and let $E$ and $F$ be  Banach spaces.

 The space of all finite rank linear operators between two Banach spaces $E$ and $F$ is denoted by ${\mathcal F}(E,F)$.

A subset $ B\subset E$    is {\it relatively $\mathcal I$-compact} if  $B \subseteq S(M)$, where $M$ is a compact subset of some Banach space $G$ and  $S \in \mathcal I(G,E)$.
A linear operator  $T:E\to F$ is {\it $\mathcal I$-compact} if $T(B_E)$ is a relatively $\mathcal I$-compact subset of $F$.
The operator ideal formed by all linear $\mathcal I$-compact operators is denoted by ${\mathcal K}_{\mathcal I}$.

Let $(\mathcal I,\|\cdot\|_{\mathcal I})$ be a Banach  operator ideal.

Following \cite{oja12} a Banach space $E$ has the {\it ${\mathcal I}$-approximation property}
if for every Banach space $F$,  $\overline{F'\otimes E}^{\|\cdot\|_{\mathcal I}}={\mathcal I}(F,E)$.
In \cite{LaTu} the norm $\|\cdot\|_{\mathcal I}$ is replaced by the operator norm $\|\cdot \|$  in ${\mathcal L}$, that is, a Banach space
  $E$ has the {\it ${\mathcal I}$-uniform approximation property}
if for every Banach space $F$,  $\overline{F'\otimes E}^{\|\cdot\|}={\mathcal I}(F,E)$.

 Note that, whenever the ideal of compact operators $\mathcal K$ is considered,  the ${\mathcal K}$-(uniform) approximation property is just the approximation property.

In particular,  $E$ has the ${\mathcal K}_{\mathcal I}$-uniform approximation property if for every Banach space $F$,  $\overline{F'\otimes E}^{\|\cdot\|}={\mathcal K}_{\mathcal I}(F,E)$,
that is, every ${\mathcal I}$-compact operator from $F$ into $E$  can be uniformly approximated by finite rank operators.

It has been proved \cite{LaTu} that a Banach space $E$ has the ${\mathcal K}_{\mathcal I}$-uniform approximation property if, and only if, the identity $Id_E\in \overline{E'\otimes E}^{\tau_{\mathcal I}} $, where $\tau_{\mathcal I}$ is the topology of uniform convergence on relatively ${\mathcal I}$-compact sets.

\section{The approximation property for Lipschitz operators.}

The above point of view of considering ideals of operator to get variants of the approximation property has motivated the study of the approximation property for metric spaces. Since the publication of  \cite{FJ09}, several ideals of Lipschitz mappings have appeared in the last years whose general approach can be found in \cite{AcRuSaYa}. Lipschitz operator ideals are the basis for considering approximation properties on metric spaces in an intrinsic way and not only on Banach spaces associated to them.

Let $X$ be  a pointed metric space (with distinguished point $0$).
A map $T:X\to E$ is {\it Lipschitz} if there is $C>0$ such that $\|T(x)-T(y)\|\leq Cd(x,y)$ for all $x,y\in X$.
The {\it Lipschitz norm} ${\rm Lip}$ is given by the infimum of all $C>0$ as above. The  space of all Lipschitz maps $T$ with $T(0)=0$ is denoted by $Lip_0(X,E)$ and forms a Banach space when endowed with the ${\rm Lip}$  norm.

Let $\chi_A$ be the characteristic function of a set $A$.
 For $x,x^{\prime }\in X$, define $ m_{xx^{\prime }}:=\chi_{\left\{ x\right\} }-\chi_{\left\{ x^{\prime }\right\} }$. We write $ \mathcal{M}(X)$ for the set of all functions $m=\sum\limits_{j=1}^{n}\lambda
_{j}m_{x_{j}x_{j}^{\prime }}$, called {\it molecules}, for any scalars $\lambda_j$ and any $x,x'$ in $X$, endowed with the norm
\begin{equation}
\left\Vert  m\right\Vert _{\mathcal{M}(X)}=\inf \left\{
\sum\limits_{j=1}^{n}\left\vert \lambda _{j}\right\vert
d(x_{j},x_{j}^{\prime }),\ \ m=\sum\limits_{j=1}^{n}\lambda
_{j}m_{x_{j}x_{j}^{\prime }}\right\} ,
\end{equation}%
where the infimum is taken over all finite representations of the molecule $m$.

\bigskip
The {\it Arens Eells space} \AE $\left( X\right) $ is the completion of the normed space $\mathcal{M}(X)$ (see \cite{A.E56}).

\bigskip
It is well-known that the map $\delta_{X} :X\rightarrow $\AE $(X)$ defined by $\delta_{X}(x)=m_{x0}$
isometrically embeds $X$ in \AE $\left( X\right) $ and that each Lipschitz operator $T:X\to E$ with values in a Banach space $E$, factors as
$$
\xymatrix{
X \ar[rr]^{T} \ar@{->}[dr]_{\delta_X} & & E, \\
& \AE(X) \, , \ar[ur]_{T_L} & }
$$
where $T_L$ is the unique continuous linear operator such that $T=T_L\circ \delta_X$.
The correspondence $T\longleftrightarrow T_{L}$ establishes an isomorphism
between the normed spaces $Lip_{0}(X,E)$ and $\mathcal{L}(\AE \left(
X\right) ,E)$.

We will consider  a {\it composition Lipschitz ideal}, i.e.,
$$
\mathcal I_{Lip}=\mathcal I\circ Lip_0
$$
 for some linear operator ideal $\mathcal I$. This equality  means that  $T\in  \mathcal I_{Lip}(X,E)$ if and only if $T=u\circ S$ for some $u\in {\mathcal I}$ and $S\in Lip_0$. In \cite{AcRuSaYa} it is proved that  $T\in  \mathcal I_{Lip}(X,E)=\mathcal I\circ Lip_0(X,E)$ if and only if $T_L \in \mathcal I(\AE(X),X)$.

In \cite{AcRuSaYa} it has been studied the approximation property for metric spaces by means of ideals of Lipschitz operators. Some of the definitions involved are natural extensions to the Lipschitz setting of the corresponding linear concepts. However, most of the proofs concerning the approximation property strongly  use linear tools and cannot be adapted to the Lipschitz case. That produces several technical difficulties on the new metric approximation theory.

Let $Y$ and $X$ be pointed metric spaces and let ${\mathcal I}$ be a linear operator ideal.
A set $K \subseteq X$ is  {\it (relatively) $\mathcal I$-Lipschitz compact} if $\delta_X(K)$ is (resp. relatively) $\mathcal I$-compact in $\AE (X)$.
If $ E$ is a Banach space, every relatively $\mathcal I$-Lipschitz compact subset of a Banach space is relatively $\mathcal I$-compact (\cite[Proposition 4.2]{AcRuSaYa}).

 We say that a Lipschitz operator $\phi:Y \to X$ is   {\it $\mathcal I$-Lipschitz compact} if $\phi(B_Y) =\phi(\{x \in Y: d(x,0) \le 1 \})$ is  relatively $\mathcal I$-Lipschitz compact.
It is known that a  linear map $T:F \to E$ between Banach spaces  is $\mathcal I$-Lipschitz compact if, and only if, it is linear $\mathcal I$-compact.

\section{The $\mathcal I$-approximation property for Lipschitz operators.}

The standard approximation property for the free spaces $\AE(X)$ has been studied by several author (e.g. \cite{dal,dal2, GoOz}). In \cite{AcRuSaYa} an approximation property has been introduced explicitly for metric spaces: the  $\mathcal I$-approximation property.

Consider  a linear operator ideal $\mathcal I$ and let  ${\mathcal I}_{Lip}={\mathcal I}\circ Lip_0$  be the associated composition Lipschitz operator ideal.

On $Lip_0(X,E)$, we consider the topology {\it Lipschitz-$\tau_{\mathcal I}$} of uniform convergence on $\mathcal I$-Lipschitz compact sets
 in the space of operators $Lip_0(X,E)$ as the one generated by the seminorms
$$
q_K(T):= \sup_{x \in K} \|T(x)\| = \sup_{m \in \delta_X(K)} \|T_L(m)\|,
$$
where $K$ is a relatively $\mathcal I$-Lipschitz compact set of $X$.

It is easy to see that this topology induces on the space $\mathcal L(F,E)$, of linear operators between Banach spaces $F$ and $E$, the topology $\tau_{\mathcal I}$ of uniform convergence on $\mathcal I$-compact sets.

 Consider a set of operators $\mathcal O(X,\AE(X)) \subseteq Lip_0(X, \AE(X))$.

\bigskip
A metric space $X$ has the {\it $\mathcal I$-Lipschitz approximation property with respect to $\mathcal O(X,\AE(X))$} if $\delta_X:X \to \AE(X)$ belongs to the Lipschitz-$\tau_{\mathcal I}$-closure of  $\mathcal O(X,\AE(X))$.

It has been proved in \cite{AcRuSaYa} that  the new concepts and results in  the Lipschitz setting recover those that come from the linear theory related to the ${\mathcal K}_{\mathcal I}$-uniform approximation property. Several connections are obtained when we consider the $\mathcal I$-Lipschitz approximation property with respect to concrete sets  $\mathcal O(X,\AE(X))$.  For instance,  the approximation property on a Banach space $E$ is recovered whenever $\mathcal O(E,\AE(E))={\mathcal F}(E,\AE(E))$ as shows the next result.

Let $E$ be a Banach space. The linearization $\beta_E:\AE(E)\to E$ of the identity map in $E$ is known as the barycentric map. It is well-known that $\beta_E(B_{\AE(E)})=B_E$.

\begin{proposition}
Let $E$ be a Banach space. If $E$ has the ${\mathcal K}$-Lipschitz approximation property with respect to ${\mathcal F}(E,\AE(E))$ then $E$ has the approximation property.
\end{proposition}

\begin{proof}
Fix $\varepsilon>0$ and a compact set $K\subset E$. By assumption, there is a finite rank operator $T\in {\mathcal F}(E,\AE(E))$ such that $\|\delta_E(x)-T(x)\|_{\AE(E)}<\varepsilon$ for all $x\in K$. Using the continuity of the barycentric map we get $\|x-\beta_E\circ T(x)\|_E<E$ for all $x\in K$. Since $\beta_E\circ T\in {\mathcal F}(E,E)$, we conclude that $E$ has the approximation property.
\end{proof}

Let $E$ be a Banach space and let ${\mathcal I}$ be an ideal of linear operators. A subset $ B\subset E$    is {\it  $\mathcal I$-bounded} if  $B \subseteq S(B_G)$, where $G$ is a  Banach space  and  $S \in \mathcal I(G,E)$. Any relatively $\mathcal I$-compact set is $\mathcal I$-bounded. If $\mathcal I=\mathcal L$ is the ideal of all continuous bounded operators, then $\mathcal I$-bounded sets are just all bounded sets.
Note that if in the definition of ${\mathcal I}$-Lipschitz approximation property we change the  topology $\tau_{\mathcal I}$ with the topology of uniform convergence on ${\mathcal I}$-bounded sets then we get the following general result in a similar way:

\begin{proposition}
Let $E$ be a Banach space and let ${\mathcal I}$ be a linear ideal. If $E$ has the ${\mathcal L}$-Lipschitz approximation property with respect to ${\mathcal I}(E,\AE(E))$ then $E$ has the ${\mathcal I}$-uniform approximation property.
\end{proposition}

The set $\mathcal O_0(X,\AE(X))=  Lip_{0\mathcal F}(X,\AE(X))$, allows to get the approximation of vector-valued Lipschitz mapping by means of finite rank Lipschitz mappings.

\begin{proposition}\cite[Proposition 4.8]{AcRuSaYa}
Let $X$ be a pointed metric  space. The following assertions are equivalent:
\begin{itemize}
\item $X$ has the ${\mathcal I}$-Lipschitz approximation property with respect to $ Lip_{0{\mathcal F}}(X,\AE(X))$.
\item For every Banach space $E$, $Lip_{0{\mathcal F}}(X,E)$ is Lipschitz-$\tau_{\mathcal I}$ dense in $Lip_0(X,E)$.
\end{itemize}
\end{proposition}

 It also permits to obtain the approximation of ${\mathcal I}$-Lipschitz compact operators  by finite finite rank Lipschitz mappings as shows the following proposition. However the converse is still an open problem.

\begin{proposition}\cite[Proposition 4.9]{AcRuSaYa}\label{aaaa}
Let $X$ be a pointed metric space and $\mathcal I$ be an operator ideal. If $X$ has the $\mathcal I$-Lipschitz approximation property with respect to $Lip_{0{\mathcal F}} (X,\AE(X))$ then, for any pointed metric space $Z$ and any $\mathcal I$-Lipschitz compact mapping $\phi:Z \to X$, the mapping  $\delta_X \circ \phi$ can be
approximated by finite rank operators of
 $Lip_{0{\mathcal F}} (Z,\AE(X))$ uniformly on $B_Z$.
\end{proposition}

When we consider the set $\mathcal O_1(E,\AE(E))=\delta_E \circ Lip_{0\mathcal F}(E, \AE(E))$ for a Banach space $E$, then  the $\mathcal I$-Lipschitz approximation property is weaker than the ${\mathcal K}_{\mathcal I}$-uniform approximation property.

\begin{proposition}\cite[Proposition 4.10]{AcRuSaYa}\label{bbbb}
Let $\mathcal I$ be an operator ideal.
Let $E$ be a Banach space with the ${\mathcal K}_{\mathcal I}$-uniform approximation property. Then $E$  has the $\mathcal I$-Lipschitz approximation property as a metric space with respect to the set $\,\,\,\delta_E \circ Lip_{0\mathcal F}(E,\AE(E))$.
\end{proposition}

The following result completes Proposition \ref{aaaa} and gives a partial converse to Proposition \ref{bbbb}.

\begin{proposition}
Let $E$ be a Banach space and $\mathcal I$ be an operator ideal. If $E$ has the $\mathcal I$-Lipschitz approximation property with respect to $Lip_{0{\mathcal F}} (E,\AE(E))$ then,  any $\mathcal I$-Lipschitz compact mapping $\phi:Z \to E$, defined on any  pointed metric space $Z$,  can be
approximated by finite rank operators in
 $\beta_E\circ Lip_{0{\mathcal F}} (Z,E)$ uniformly on $B_Z$.
\end{proposition}

\begin{proof}
Let   $Z$ be a pointed metric space and let $\phi:Z \to E$ be a $\mathcal I$-Lipschitz compact mapping. Fix $\varepsilon>0$. By Proposition \ref{aaaa}, there is  $f$ in $Lip_{0{\mathcal F}} (Z,E)$ such that $\|f(x)-\delta_E\circ \phi(x)\|_{\AE (E)}<\varepsilon$ for all $x\in B_Z$. Using the continuity of the linear barycentric map $\beta_E$ we get  that $\|\beta_E\circ f(x)- \phi(x)\|_{E}<\varepsilon$ for all $x\in B_Z$.
\end{proof}

If $\mathcal O_2(X,\AE(X))= \mathcal F \circ \delta_X(X, \AE(X))$ then the connection is summarized in the following proposition.

\begin{proposition}\cite[Proposition 4.11]{AcRuSaYa}
Let $\mathcal I$ be an operator ideal and let $X$  be a pointed metric space.
 If $\AE(X)$ has the ${\mathcal K}_{\mathcal I}$-uniform approximation property, then
 $X$ has  the $\mathcal I$-Lipschitz approximation property with respect to the class $\mathcal F \circ \delta_X(X,\AE(X))$.
 \end{proposition}

\end{document}